\documentclass{article}%
\usepackage{amsfonts}
\usepackage{graphicx,epstopdf}
\usepackage{amsmath,mathrsfs}
\usepackage{amssymb}%
\usepackage{color}
\providecommand{\U}[1]{\protect\rule{.1in}{.1in}}
\newtheorem {theorem}{Theorem}[section]

\newtheorem {corollary}{Corollary}[section]

\newtheorem{lemma}{Lemma}[section]

\newtheorem{remark}{Remark}[section]

\newenvironment{proof}[1][Proof]{\textbf{#1.} }{\
\rule{0.5em}{0.5em}}
\def\P{{\mathbb P}}

\begin{document}

\begin{center}
{\LARGE Exact Moderate and Large Deviations for Linear Random Fields}

\bigskip
  
\centerline{\today}

\bigskip Hailin Sang$^{a}$
and Yimin Xiao$^{b}$

\bigskip$^{a}$ Department of Mathematics, The University of Mississippi,
University, MS 38677, USA. E-mail: sang@olemiss.edu

\bigskip$^{b}$ Department of Statistics and Probability, Michigan State University, East
Lansing, MI 48824, USA. E-mail:  xiao@stt.msu.edu
\end{center}

\bigskip \textbf{Abbreviated Title: }{\Large Deviations for linear random fields}

\begin{center}
\bigskip
\textbf{Abstract}
\end{center}
By extending the methods in Peligrad et al. (2014a, b), we establish exact moderate and large deviation
asymptotics for linear random fields with independent innovations. These results are useful for studying
nonparametric regression with random field errors and strong limit theorems. \\

\noindent Key words: Large deviation, moderate deviation, linear random fields,
nonparametric regression,
Davis-Gut law of the iterated logarithm.\\

\noindent  {\textit{MSC 2010 subject classification}: 60F10, 60G60, 62E20}


\section{Introduction}

Random fields play a central role in modeling and analyzing spatially correlated
data and have a wide range of applications.
As a consequence, there has been increasing interest in studying
them in probability and statistics. 

Consider a linear random field $X = \{X_{j,k}, \, (j,k)\in \mathbb{Z}^2\}$
defined on a probability space $(\Omega, {\mathcal F}, \mathbb P)$ by
\begin{equation}\label{rf}
X_{j,k}=\sum_{r, s\in \mathbb{Z}} a_{r,s}\xi_{j-r, k-s},
\end{equation}
where $\{a_{r,s}, (r, s) \in \mathbb{Z}^2\}$ is a square summable sequence
of constants and the innovations $\{\xi_{r,s}, (r, s) \in \mathbb{Z}^2\}$ and
$\xi_0$ are i.i.d. random variables with $\mathbb{E}\xi_0=0$
and $\mathbb{E}\xi^2_0=1$. Under these conditions, $X_{j,k}$ in (\ref{rf}) is well-defined 
because the series in the right-hand side of \eqref{rf} converges in the 
$L^2(\Omega, \mathbb P)$-sense and almost surely. See Lemma \ref{Lem:Conv} in the
Appendix. In the literature, there have been extensive studies
on limit theorems and estimation problems for linear random fields. For example,
Marinucci and Poghosyan (2001), and Paulauskas (2010) studied the asymptotics
for linear random fields, including law of large numbers, central limit theorems
and invariance principles,  by applying the Beveridge-Nelson decomposition method.
Banys et al. (2010) applied ergocic theory to study strong law of large numbers for 
linear random fields. Mallik and Woodroofe (2011) also established
the central limit theorem for linear random fields, and their method does not rely on
the Beveridge-Nelson decomposition. Under various settings, Tran (1990), Hallin et al.
(2004a, 2004b), El Machkoui (2007, 2014), El Machkouri and Stoica (2010), 
and Wang and Woodroofe (2014) studied local
linear regression, kernel density estimation and their asymptotics for linear
random fields. Gu and Tran (2009) developed fixed design regression study for
negatively associated random fields.

However, few authors have studied moderate and large deviations for linear random fields.
Davis and Hsing (1995), Mikosch and  Samorodnitsky (2000), Mikosch and Wintenberger (2013)
established large deviation results for certain stationary sequences, including linear processes 
with short-range dependence.  For linear processes which allow long range dependence, we 
mention that Djellout and Guillin (2001) proved moderate and large deviation results for linear 
processes with i.i.d. and bounded innovations; Djellout et al. (2006) studied moderate deviation
estimate for the empirical periodogram of a linear process; Wu and Zhao (2008) obtained
moderate deviations for stationary causal processes and their main theorem can be applied
to functionals of linear processes; and more recently, Peligrad et al. (2014a, b) established
exact moderate and large deviation asymptotics 
for linear processes with independent innovations. 

The main purpose of this paper is to extend the method in Peligrad et al. (2014a, b)
to establish exact moderate and large deviations for linear random fields as in
(\ref{rf}).  Let $\{\Gamma_n\}$ be a sequence of finite subsets of $\mathbb{Z}^2$ and 
denote the cardinality of $\Gamma_n$ by $|\Gamma_n|$. To be specific, we can take 
$\Gamma_n = [-n, n]^2 \cap \mathbb{Z}^2$, or $[1, n]^2 \cap \mathbb{Z}^2$ or more 
general rectangles. Define $S_n:=
S_{\Gamma_n}:=\sum_{(j,k)\in\Gamma_n} X_{j,k}$. By Lemma \ref{Lem:Conv} in the
Appendix, it can be written as 
\begin{equation}\label{Def:S}
S_n=\sum_{r,s\in\mathbb{Z}} b_{n,r,s}\xi_{-r, -s},
\end{equation}
where $b_{n,r,s}=\sum_{(j,k)\in\Gamma_n} a_{j+r, k+s}$. Let $\sigma_n^2
=\mathbb{E}S_n^2$.
The main results of this paper, Theorems 2.1 - 2.3, quantify the
roles of the moment and right-tail properties of $\xi_0$, the magnitude
of the coefficients $\{a_{r,s}\}$, as well as the speed of convergence
of $x_n\rightarrow \infty$, in the moderate and large deviation probabilities
for $\mathbb{P}\left(S_{n}\geq x_{n}\sigma_{n}\right)$. These results are
useful for studying asymptotic properties and statistical inference of linear
random fields. As examples, we show that our
moderate and large deviation results can be applied for studying nonparametric
regressions and for obtaining convergence rate in the law of the iterated
logarithm of linear random fields.



For simplicity of presentation, we focus on linear random fields
indexed by $\mathbb{Z}^2$. The theorems presented in this paper can be easily
extended to linear random field $X_{\bf j}=\sum_{{\bf r}\in \mathbb{Z}^N}
a_{\bf r}\xi_{\bf j-r}$ on $\mathbb{Z}^N$ with $N\ge 3$.

In this paper we shall use the following notations. For any constant $ p\ge 1$, we define 
$\|a\|_p:=\left[\sum_{r,s\in \mathbb{Z}}|a_{r,s}|^p\right]^{1/p}$. 
Then $\|a\|_2<\infty$ by the assumption and $\|a\|_p$ 
may be finite for some values of $p<2$. Similarly, for a random variable $\xi$, we use $\|\xi\|_p$
to denote its $L^p(\mathbb P)$-norm for $p \ge 1$. Let $\Phi(x)$
be the distribution function of the standard normal random variable.

For two sequences $\{a_n\}$ and $\{b_n\}$ of real numbers,  
$a_{n}\mathbb{\sim}b_{n}$ means $a_{n}/b_{n}\rightarrow1$ as $n\rightarrow \infty$;
$a_{n}\propto b_n$ means that $a_{n}/b_{n}\rightarrow C$ as $n\rightarrow \infty$
for some constant $C>0$; for positive sequences, the notation $a_{n}\ll b_{n}$
or $b_n\gg a_n$ replaces Vinogradov symbol $O$ and they mean that $a_{n}/b_{n}$
is bounded; $\lceil x\rceil$ means the smallest integer which is greater than or equal to $x$.

The rest of this paper has the following structure. Section 2 gives the main results 
on moderate and large deviations for $S_n$ in \eqref{Def:S}. In
Section 3 we apply the main results to nonparametric regression estimates and
prove a Davis-Gut law of iterated logarithm for linear random fields. The Appendix
provides the existing results which are useful for  proving the theorems
in Section 2. \\

\textbf{Acknowledgement}  The authors thank the referee and the Associate 
Editor for their careful reading of the manuscript and for their insightful comments, 
which have helped to improve the quality of this paper. The research of Yimin Xiao
is partially supported by NSF grants DMS-1612885 and DMS-1607089.


\section{Main results}

Even though the double sum $S_n$ in \eqref{Def:S} can be written (in infinitely many ways) 
as a single weighted sum of infinitely many i.i.d. random variables indexed by non-negative 
integers,\footnote{For example, 
$$ 
S_n=\sum_{i=0, |r|=i, |s|\le i \; \hbox{\tiny or }  |s|=i, |r|< i }^\infty b_{n,r,s}\xi_{-r, -s},
$$
where the number of terms for each index
$i$ is finite and the summation of the terms with the same $i$ can be taken in any order.}
the important role of the configuration of $\Gamma_n$ is usually hidden in 
such a representation and a partial order in $\mathbb{Z}^2$, which may not be natural for 
the problem under investigation, has to be imposed. These make it difficult to solve the problems 
for random fields satisfactorily by applying directly the results on a weighted sum of random 
variables indexed by one variable. Quite often new methods have to be developed. We 
refer to Chapter 1 of Klesov (2014) for further illustrations on connections as well as significant
differences in limit theorems of random fields and stochastic processes of one-variable.   

The objective of this section is to study moderate and large deviations for 
the partial sum $S_n$ in \eqref{Def:S} by extending the results for linear processes 
as in Peligrad et al. (2014a, b) to linear random fields. 

We will need some notations. Define 
$$D_{nt}:=\sum_{r,s\in\mathbb{Z}}|b_{n,r,s}|^t; \qquad 
U_{nt}:=(D_{n2})^{-t/2}D_{nt}.$$
Then $\sigma^2_n :=\mathbb{E}(S_n^2)=D_{n2}$. 
To avoid degeneracy, we assume tacitly $\sigma_n > 0$ for every $n$. 
Let $\rho_n^2:=\max_{r,s\in\mathbb{Z}}b_{n,r,s}^2/\sigma^2_n$. We will assume that $\rho_n^2 \to 0$ 
which means that the contribution of any single coefficient $b_{n,r,s}$ is negligible compared with 
$\sigma^2_n$. We remark that the magnitudes of $b_{n,r,s}^2$ and $D_{nt}$ depend on the coefficients 
$\{a_{r,s}, (r,s)\in \mathbb Z^2\}$ and the configuration of $\Gamma_n$. An interesting case is when 
$\{a_{r,s}\}$ is isotropic and $\Gamma_n = [1, n]^2\cap \mathbb Z^2$. See (\ref{la}) and Lemma \ref{Dorder}
below for details. More generally, the case when $\{a_{r,s}\}$ is anisotropic (i.e., $|a_{r, s}|$ depends $r$ 
and $s$ at different rates) and $\Gamma_n$ is a rectangle in $\mathbb Z^2$ 
can also be considered.

Our first theorem  is the following moderate deviation result.
\begin{theorem}\label{ModerateD}
Assume that the random variable $\xi_0$ satisfies $\|\xi_0\|_p<\infty$
for some $p>2$ and $\rho_n^2\rightarrow 0$ as $n\rightarrow \infty$.
Then for $x_n\ge 0$, $x_n^2\le 2\ln (U_{np}^{-1})$,
the moderate deviation result holds:
\begin{equation}\label{Eq:Th21a}
\mathbb{P}\left( S_{n}\geq x_{n}\sigma_{n}\right)  =(1-\Phi(x_{n}%
))(1+o(1))\text{ as }n\rightarrow\infty; 
\end{equation}
\begin{equation}\label{Eq:Th21b}
\mathbb{P}\left(  S_{n}\leq -x_{n}\sigma_{n}\right)  =(1-\Phi(x_{n}%
))(1+o(1))\text{ as }n\rightarrow\infty. 
\end{equation}
\end{theorem}
\begin{proof}
We only need to prove the statement (\ref{Eq:Th21a}). The proof is a
modification of that of Corollary 3, part (iii) of Peligrad et al. (2014a), which
is given in the Supplementary Material, Peligrad et al. (2014b).
The main idea is to applying
Theorem \ref{frolov} in the Appendix for triangular arrays.
For this purpose, we decompose the partial sum $S_n$ as $S_n = M_n + R_n$,
where $M_n=\sum_{|r|\le k_n} \sum_{|s|\le k_n} b_{n,r,s}\xi_{-r, -s}$ for
some integer $k_n$ which will be chosen later and $R_n$ is the remainder
$$
R_n= \bigg(\sum_{|r|> k_n}
\sum_{s\in \mathbb{Z}} +\sum_{|r|\le k_n}\sum_{|s|> k_n} \bigg)b_{n,r,s}\xi_{-r, -s}.$$
By the Cauchy-Schwarz inequality, we have
\begin{equation*}
b_{n,r,s}^{2}\leq |\Gamma_n|\sum_{(j,k)\in \Gamma_n} a^2_{j+r, k+s}
\end{equation*}
and then,
\begin{equation} \label{Eq:tailsum1}
\sum_{r, s\in \mathbb{Z}}b_{n,r,s}^{2}
\le |\Gamma_n|\sum_{r, s\in \mathbb{Z}} \sum_{(j,k)\in \Gamma_n} a^2_{j+r, k+s}
= |\Gamma_n|^2\sum_{r, s\in \mathbb{Z}}a_{r,s}^{2}. 
\end{equation}
In the above, we have applied Fubini's theorem to change the order of summations.
Therefore, for every integer $n\ge 1$, we have $\sum_{r, s\in \mathbb{Z}}
b_{n,r,s}^{2}<\infty$ which yields $\sum_{r, s\in \mathbb{Z}}
|b_{n,r,s}|^{p}<\infty$ since $p>2$.

By applying Rosenthal's inequality (cf. de la Pe\~na and Gin\'e,
1999) to $R_n$, we see that there is a constant
$C_p$ such that
\begin{align*}
\mathbb{E} \big(|R_{n}|^{p} \big)&\leq C_{p} \Bigg[\bigg( \Big(\sum_{|r|> k_n}
\sum_{s\in \mathbb{Z}} +\sum_{|r|\le k_n}\sum_{|s|> k_n} \Big)b_{n,r,s}^{2}\bigg)^{p/2}\\
& \qquad \qquad +\mathbb{E} (|\xi_{0}|^{p})\bigg(\sum_{|r|> k_n}\sum_{s\in \mathbb{Z}}
+\sum_{|r|\le k_n}\sum_{|s|> k_n} \bigg) |b_{n,r,s}|^{p}\Bigg]. 
\end{align*}
Now for each positive integer $n$, we select integer $k_{n}$ 
large enough such that
\[
\bigg(\sum_{|r|> k_n}\sum_{s\in \mathbb{Z}} +\sum_{|r|\le k_n}\sum_{|s|> k_n} \bigg)
b_{n,r,s}^{2}
\leq \|\xi_{0}\|_{p}^{2} \bigg(\sum_{r,s\in \mathbb{Z}} |b_{n,r,s}|^{p}
\bigg)^{2/p}.
\]
This is possible because of (\ref{Eq:tailsum1}) and the fact that 
 $\sum_{r, s \in \mathbb Z} a_{r,s}^2< \infty$.

With the above selection of $k_n$, we obtain
\begin{equation}\label{p-rest}
\mathbb{E} \big(|R_{n}|^{p} \big)\leq 2C_{p}\mathbb{E}(|\xi_{0}|^{p})\sum_{r,s\in \mathbb{Z}} |b_{n,r,s}|^{p}.
\end{equation}
Similarly, we can verify that $M_n$ and thus $S_n$  also have finite moments of order $p$.

Since $M_n$ is the sum of $(2k_{n}+1)^2$ independent random variables, we can view
$S_{n} = M_n + R_n$ as the sum of $(2k_{n}+1)^2+1$ independent random variables
and apply Theorem \ref{frolov} to prove (\ref{Eq:Th21a}). As in the
proof of part (iii) of Corollary 2.3 in Peligrad et al. (2014b), we use
(\ref{p-rest}) to derive
\begin{equation*}
\begin{split}
L_{np}&= \frac1 {\sigma_n^p} \bigg(\sum_{|r|\le k_n} \sum_{|s|\le k_n}
\mathbb E \big[(b_{n,r,s}\xi_{-r, -s})^p I(b_{n,r,s}\xi_{-r, -s}\ge 0)\big] +
\mathbb E \big[R_n^p I(R_n\ge 0)\big]\bigg)\\
&\le (2C_p + 1)\mathbb{E}(|\xi_{0}|^{p})U_{np}.
\end{split}
\end{equation*}
Since $\rho_n^2 \to 0$ and $p > 2$ imply $U_{np} \to 0$ as $n\to \infty$,
we have  $L_{np}\rightarrow0$. Similarly, for all $x \ge 0$ such that
$ x^2 \le 2 \ln (U_{np}^{-1})$, we can verify that $\Lambda_{n}(x^{4},x^{5},
\varepsilon)\rightarrow0$ for any $\varepsilon>0$, and  $x^{2}
-2\ln(L_{np}^{-1})-(p-1)\ln\ln(L_{np}^{-1})\rightarrow-\infty$, as
$n \to \infty$. Hence the conditions of Theorem
\ref{frolov} are satisfied, and (\ref{Eq:Th21a}) follows.
\end{proof}

\begin{remark}\label{remark1}
The condition $\rho_n\rightarrow 0$ in Theorem \ref{ModerateD} and the following
Theorems \ref{LD} and \ref{MDLD} can be replaced by suitable conditions on
$|\Gamma_n|$ and $\sigma_n^2$ that are easier to verify. By H\"older's
inequality, we have $\rho_n\le \frac{\|a\|_u|\Gamma_n|^{1/v}} {\sigma_n}$,
where $1\le u\le 2$ and $v$ is the conjugate of $u$, $1/u+1/v=1$.
Therefore, we can replace the condition $\rho_n\rightarrow 0$ by
$\|a\|_u<\infty$ and $\frac{|\Gamma_n|^{1/v}}{\sigma_n}\rightarrow 0$.
In particular, if $\|a\|_1 < \infty$ which is
the short range dependence case, then $\rho_n\le \|a\|_1/\sigma_n$.
In this case, we can replace the condition $\rho_n\rightarrow 0$ by 
$\sigma_n\rightarrow \infty$ (as a consequence, we also have $|\Gamma_n|\rightarrow \infty$). 
See Mallik and Woodroofe (2011) for more information on bounds for $\rho_n$. If $\Gamma_n$
is a union of $l$ finitely many discrete rectangles, by Proposition 2 of
the same paper, $\rho_n\le 20\left(\frac{\sqrt{l}
\|a\|_2}{\sigma_n}\right)^{1/5}+\frac{8\sqrt{l}\|a\|_2}{\sigma_n}$. Therefore,
in this case, we can replace
the condition $\rho_n\rightarrow 0$ by $\|a\|_2<\infty$ and $\sigma_n\rightarrow
\infty$.
\end{remark}



Next, we study precise large deviations for the partial sums $S_n$ defined in (\ref{Def:S}).   We
will focus only on the case when  $\xi_0$ has a right regularly varying tail (see Remark 2.2 below for 
information on other interesting cases). More precisely, we assume 
that there is a constant $t>2$ such that
\begin{equation}\label{tail1}
\mathbb{P}(\xi_{0}\geq x)=\frac{h(x)}{x^{t}},\;\; \
\hbox{ as }\, x\rightarrow\infty.
\end{equation}
Here $h(x)$ is a slowly varying function at infinity. 
Namely, $h(x)$  is a measurable positive function satisfying $\lim_{x\rightarrow\infty}h(\lambda x)/h(x)=1$ 
for all constants $\lambda>0$.  Bingham et al. (1987) or Seneta (1976) provide systematic accounts on 
regularly varying functions. For reader's convenience, we collect some useful properties of slowly varying 
functions in  Lemma \ref{Karamata} in the Appendix.



Notice that condition \eqref{tail1}  is an assumption on the right tail of $\xi_0$.  
The left tail of $\xi_0$ can be arbitrary. In particular, it implies that  $\xi_0$ does not 
have $p$-th  moments  for $p>t$, and it may or may not have $p$-th moments for $p<t$.

The notion of regular variation such as defined in (\ref{tail1})  is closely related to large 
deviation results for sums of random variables or processes.
Such results have been proved by A.V. Nagaev (1969a, b) and S.V. Nagaev (1979) for partial 
sums of i.i.d. random variables, and have been extended to partial sums of certain stationary 
sequences by Davis and Hsing (1995), Mikosch and Samorodnitsky (2000), Mikosch and 
Wintenberger (2013), and Peligrad et al. (2014a, b). 

We now comment briefly on the connections and differences of the results and methods in the 
aforementioned references to those in the present paper. The approach of Davis and Hsing (1995) 
is based on weak convergence of 
point processes  and the link between the large deviation probability  and the asymptotic behavior 
of extremes.  As shown in Example 5.5 in Davis and Hsing (1995), their results are applicable to
a class of linear processes with short-range dependence. Moreover, as pointed out by  Mikosch 
and Wintenberger (2013, p.853), the method of Davis and Hsing (1995) could not be extended to 
the case of $t \ge 2$.

Mikosch and Samorodnitsky (2000) studied precise large deviation results  for a class  
linear processes with a negative drift. More specifically, they consider 
$$
X_n = -\mu + \sum_{j \in \mathbb Z} \varphi_{n-j} \varepsilon_j,  \ \ \ n \in {\mathbb Z},
$$
where $\mu > 0$ is a constant, $\{\varepsilon_j\}$ are i.i.d. innovations that satisfy a 
two-sided version of \eqref{tail1} and the coefficients $\{\varphi_{j}\}$ satisfy 
$\sum_{j \in \mathbb Z}|j \varphi_j| < \infty$. 
In particular, the process $\{X_n, n \in \mathbb Z\}$ is short-range dependent. 

Mikosch and Wintenberger (2013) established precise large deviation results for a stationary sequence 
$\{X_n, n \in \mathbb Z\}$ that satisfies the following (and some other technical) conditions:
(i) All finite dimensional distributions of $\{X_n, n \in \mathbb Z\}$ are regularly varying 
with the same index $\alpha$; and (ii) The anti-clustering conditions. See Mikosch and Wintenberger (2013)
for precise descriptions of these conditions. We remark that, even though their methods and results cover
a wide class of stationary sequences,  the condition (i) is a lot 
stronger than \eqref{tail1} and is not easy to verify for a general linear process.  Moreover, 
as pointed out by Mikosch and Wintenberger  (2013, page 856),  the condition (ii) excludes stationary 
sequences with ``long range dependencies of extremes". 

 
We believe that it would be interesting from both theoretical and application viewpoints to extend the large 
deviation results in Davis and Hsing (1995), Mikosch and Samorodnitsky (2000), Mikosch and 
Wintenberger (2013), and Peligrad et al. (2014a, b) to stationary random fields. The present paper 
is one step towards this direction. More specifically, we follow the approach of Peligrad et al. (2014a, b) and 
prove the following precise large deviation theorem, which is applicable to linear random fields with 
long-range dependence.
\begin{theorem}\label{LD}
Assume that  $\{b_{n,r,s}, r, s \in \mathbb Z\}$ is a sequence of positive
numbers with $\rho_n^2\rightarrow 0$ as $n\rightarrow \infty$ and  $\xi_0$ satisfies
condition (\ref{tail1}) for certain constant $t>2$.
For $x=x_{n}\geq C_t[\ln(U_{nt}^{-1})]^{1/2}$, where
$C_t>e^{t/2}(t+2)/\sqrt{2}$ is a constant, we have
\begin{equation}\label{Eq:LD2-2}
\begin{split}
\mathbb{P} \big(S_{n}\geq x \big)&= \big(1+o(1)\big)\sum_{r,s\in\mathbb{Z}}
\mathbb{P} \big(b_{n,r,s}\xi_{-r, -s}\geq
x \big)\\
&=  \big(1+o(1)\big) x^{-t}  \sum_{r,s\in\mathbb{Z}} b_{n,r,s}^t h\Big(\frac{x}{b_{n,r,s}}\Big),\ \ \ \text{ as }\ n\rightarrow\infty. 
\end{split}
\end{equation}
\end{theorem}
\begin{proof}
Since the second equality in (\ref{Eq:LD2-2}) follows directly from the first and (\ref{tail1}),
we only need to prove the first equality. The proof is essentially a modification of  that of Theorem 2.2 in Peligrad
et al. (2014b), by replacing the quantities $c_{ni}$ there by $b_{n,r,s}$,
and the sum $\sum_{i=1}^{ k_n}$ by the double sum $\sum_{r,s\in \mathbb{Z}}$.
A somewhat new ingredient for the proof is
to use a new version of the Fuk-Nagaev inequality for the double sums of infinitely
many random variables which is stated as Theorem \ref{FN} in the Appendix.  Hence we
will only sketch the main steps of the proof.

Without loss of generality, we normalize the partial sum $S_n$ by its variance
and assume
\begin{equation}
\sum_{r,s\in\mathbb{Z}}b_{n,r,s}^2=1\ \text{ and }\
\rho_n^2=\max_{r,s \in\mathbb{Z}}b_{n,r,s}^{2}\rightarrow0 \ \ \text{ as }
n\rightarrow\infty. \label{sum1}
\end{equation}
Then, for any constant $t > 2$, we have $U_{nt} = D_{nt}$ and $D_{nt} = \sum_{r,s\in\mathbb{Z}}
b_{n,r,s}^t \leq\max_{r,s \in\mathbb{Z}} b_{n,r,s}^{t-2}
\rightarrow0$, which implies that  $D_{nt}^{-1}\rightarrow\infty.$
Moreover, the sequence $S_n$ is stochastically
bounded (i.e., $\lim_{K\rightarrow\infty}\sup
_{n}\mathbb{P}(|S_{n}|>K)=0$) since $\mathbb{E} (S_n^2)=1$.
By following the proofs of Lemma 4.1 and Proposition 4.1 in Peligrad et al. (2014b),  
 for any $0<\eta<1,$
and $\varepsilon>0$ such that $1-\eta>\varepsilon$ and any $x_n \to \infty$, we have
\begin{equation} \label{ineqprop}
\begin{split}
&\bigg|\mathbb{P}(S_{n}\geq x_{n})-\mathbb{P}(S_{n}^{(\varepsilon x_{n})}\geq
x_{n})-\sum_{r,s\in\mathbb{Z}}\mathbb{P}(b_{n,r,s}\xi_{-r, -s}\geq(1-\eta)x_{n})\bigg|\\
& \leq
o(1)\sum_{r,s\in\mathbb{Z}}\mathbb{P}(b_{n,r,s}\xi_{-r, -s}\geq\varepsilon x_{n}) \\
&\qquad +
\sum_{r,s\in\mathbb{Z}}\mathbb{P}((1-\eta)x_{n}\leq b_{n,r,s}\xi_{-r, -s}
<(1+\eta)x_{n}),
\end{split}
\end{equation}
where  $S_{n}^{(\varepsilon x_{n})}=\sum_{r,s\in\mathbb{Z}}b_{n,r,s}
\xi_{-r, -s}I(b_{n,r,s}\xi_{-r, -s}<\varepsilon x_{n})$, $o(1)$ depends
on the sequence $x_{n},$ $\eta$ and $\varepsilon$ and
converges to $0$ as $n\rightarrow\infty.$
See also Lemma 4.2 and Remark 4.1 in Peligrad et al. (2014b) for sums of 
infinite many random variables. 

As in the proof of Theorem 2.2 in Peligrad et al. (2014b), by analyzing
the two terms of the right-hand side and the last term of the left-hand side
of \eqref{ineqprop}, we  derive that for
any fixed $\varepsilon>0$,
\begin{equation} \label{ineqLMD}
\mathbb{P}(S_{n}\geq x)= \big(1+o(1)\big) \sum_{r,s\in\mathbb{Z}}\mathbb{P}(b_{n,r,s}
\xi_{-r, -s}\geq x)+\mathbb{P}(S_{n}^{(\varepsilon x)}\geq x)
\end{equation}
as $n\rightarrow\infty$. It remains to show that the term 
$\mathbb{P}(S_{n}^{(\varepsilon x)} \geq x)$ is negligible compared with the 
first term in (\ref{ineqLMD}). To this end,
we apply Theorem \ref{FN} to the sequence $\{b_{n,r,s}
\xi_{-r, -s}, \, r, s \in \mathbb Z\}$ with $y = \varepsilon x$ to
derive that for any constant
$m >t$,
\begin{equation}\label{toshow_1}
\mathbb{P} \big(S_{n}^{(\varepsilon x)}\geq x \big)\leq
\exp\bigg(-\frac{\alpha^{2}x^{2}}
{2e^{m}} \bigg)+ \bigg( \frac{A_{n}(m;0,\varepsilon x)}
{\beta\varepsilon^{m-1}x^{m}}\bigg)^{\beta/\varepsilon},
\end{equation}
where $\beta=m/(m+2),$ $\alpha=1-\beta=2/(m+2)$ and we have used the fact that
$B^2_n(- \infty, \varepsilon x) \le 1$, which follows from (\ref{sum1}).

Then, following the proof of Theorem 2.2 in Peligrad et al. (2014b),
we can show that, for all
$x=x_{n}\geq C_t[\ln(U_{nt}^{-1})]^{1/2}$, where
$C_t>e^{t/2}(t+2)/\sqrt{2}$ is a constant, we have
\begin{equation} \label{toshow}
\begin{split}
\exp\biggl(-\frac{\alpha^{2}x^{2}}{2e^{m}}\biggl) &+\left(\frac{A_{n}(m;0,\varepsilon
x)}{\beta\varepsilon^{m-1}x^{m}}\right)  ^{\beta/\varepsilon}=o(1)\sum_{r,s\in\mathbb{Z}}
\frac{b_{n,r,s}^{t}}{x^{t}}h\bigg(\frac{x}{b_{n,r,s}}\bigg) \\
&=o(1)\sum_{r,s\in\mathbb{Z}} \sum_{r,s\in\mathbb{Z}}\mathbb{P}(b_{n,r,s}
\xi_{-r, -s}\geq x)\ \text{ as
}  \ n\rightarrow\infty.
\end{split}
\end{equation}
In particular, we use the observation
\begin{equation}
D_{nt}=\sum_{r,s\in \mathbb{Z}}b_{n,r,s}^{2\eta}b_{n,r,s}^{t-2\eta
}\leq \bigg(\sum_{r,s\in \mathbb{Z}}b_{n,r,s}^{2}\bigg)^{\eta} \bigg(\sum_{r,s\in \mathbb{Z}}b_{n,r,s}
^{(t-2\eta)/(1-\eta)}\bigg)^{1-\eta}. \label{rel-cnt}%
\end{equation}
Here we have omitted the details for deriving (\ref{toshow}) as it is
very similar to the proof in Peligrad et al. (2014b). Finally, by combining
(\ref{ineqLMD}),  (\ref{toshow}) and (\ref{toshow_1}),
we obtain the first equality in (\ref{Eq:LD2-2}). This completes the  proof 
of Theorem \ref{LD}.
\end{proof}


\begin{remark}
Besides the case of regularly varying tails such as \eqref{tail1}, large deviation results for sums of
independent random variables or linear procesess have been studied by several authors under 
the following two  conditions, respectively: (a)  $\xi_0$ satisfies the Cram\'er condition:   
there exists a constant $h_0 > 0$ such that 
$\mathbb E(e^{h \xi_0}) < \infty$ for $|h|\le h_0$;  
(b)  $\xi_0$ satisfies the Linnik condition: there is a constant $\gamma \in (0, 1)$ such that 
$\mathbb E(e^{| \xi_0|^\gamma}) < \infty$.  See Nagaev (1979), Jiang, et al. (1995), 
Saulis and Statulevi\u cius (2000), Djellout and Guillin (2001),  
Ghosh and Samorodnitsky (2008), Li, et al. (2009), among others.

In light of Theorem \ref{LD} and the above discussions, we think it would be interesting 
to study the following problems: 
\begin{itemize}
\item Study precise large deviation problems for linear random fields under the Cram\'er and Linnik
conditions. 
 
\item Extend the methods of  Davis and Hsing (1995), Mikosch and Samorodnitsky (2000), 
Mikosch and Wintenberger (2013) to establish large deviation results for  stationary random fields.
\end{itemize}
\end{remark}

Notice that the tail conditions of Theorems \ref{ModerateD} and \ref{LD}
are different since one involves the moment and the other just involves
the right tail behavior. Put these conditions together,
we have the following tail probabilities over all $x_n\ge c$ for some $c>0$. This theorem is a natural extension of 
the uniform moderate and large deviations for sums of i.i.d. random variables (cf. Theorem 1.9 in  S.V. Nagaev, 1979). 
\begin{theorem}\label{MDLD}
\label{mix}
Assume that $\xi_0$ satisfies $\|\xi_0\|_p<\infty$ for some $p>2$ and
the right tail condition (\ref{tail1}) for some constant $t>2$. Assume
also that $b_{n,r,s}>0$ and $\rho_n^2 \rightarrow 0$ as $n\rightarrow
\infty$. Let $(x_{n})_{n\ge1}$ be any sequence such that for some $c>0$
we have $x_{n}\geq c$ for all $n$. Then, as
$n\rightarrow\infty$,
\begin{equation}
\mathbb{P}\left(  S_{n}\geq x_{n}\sigma_{n}\right)
=(1+o(1))\bigg[x_n^{-t}  \sum_{r,s\in\mathbb{Z}} b_{n,r,s}^t h\Big(\frac{x_n}{b_{n,r,s}}\Big)+1-\Phi(x_{n})\bigg].
\label{MD+LD}%
\end{equation}
\end{theorem}
\begin{proof}
The proof is a modification of that of Theorem 2.1 and Corollary 2.3,
part (i), in Peligrad et al. (2014b). We sketch the proof here for completeness. Without loss of generality we
may assume $2<p<t$. Let $x=x_{n}\rightarrow \infty.$ For simplicity,
we assume (\ref{sum1}). Under the condition in this theorem, as in
the proof of Theorem \ref{LD}, we have that (\ref{ineqLMD}) holds.
Denote
\[
X_{n, r,s}^{\prime}=b_{n,r,s}\xi_{-r,-s}I(b_{n,r,s}\xi_{-r,-s}\leq\varepsilon x).
\]

We now apply Lemma \ref{frolov-trunc} to the second term
in the right-hand side of (\ref{ineqLMD}). To this end, we decompose
the sum $S_n^{(\varepsilon x)}$ as a finite sum, i.e., the sum of $\sum_{|r|\le k_n}
\sum_{|s|\le k_n} X_{n, r,s}^{\prime}$ with $(2k_n+1)^2$ terms for some
$k_n$ and the remainder
$$
R_n'= \bigg(\sum_{|r|> k_n}\sum_{s\in \mathbb{Z}} +
\sum_{|r|\le k_n}\sum_{|s|> k_n} \bigg)
X_{n, r,s}^{\prime}.
$$
By Rosenthal's inequality (cf. de la Pe\~na and Gin\'e, 1999), it is
easy to derive
\[
\mathbb{E}|R_{n}^{\prime}|^{p}\leq2C_{p}^{\prime}\mathbb{E}|\xi_{0}|^{p}%
\sum_{r,s}|b_{n,r,s}|^{p}
\]
for some constant $C_p'$. Consequently, the quantity $L_{np}$ in
Lemma \ref{frolov-trunc} is bounded by
\[
L_{np}\leq(2C_{p}^{\prime}+1)\mathbb{E}|\xi_{0}|^{p}
\sum_{r,s}|b_{n,r,s}|^{p}=(2C_{p}^{\prime}+1)D_{np}\mathbb{E}|\xi_{0}|^{p}.
\]
See also the proof of Corollary 2.3, part (i), in Peligrad et al. (2014b).
Then, by Lemma \ref{frolov-trunc} if $x^{2}\leq c\ln((2C_{p}^{\prime
}+1)D_{np}\mathbb{E}|\xi_{0}|^{p})^{-1}$ for $c<1/\varepsilon$, we have
$x^{2}\leq c\ln(L_{np}^{-1})$ for $c<1/\varepsilon$ and
\begin{equation}\label{truncnorm}
\mathbb{P}(S_n^{(\varepsilon x)}\geq x)  =(1-\Phi(x))(1+o(1)).
\end{equation}
Notice that (\ref{truncnorm}) also holds for $x^{2}\leq c\ln(D_{np})^{-1}$
for any $c<1/\varepsilon$
and large enough $n$ since $D_{np}\rightarrow0$. Recall that $2<p<t$.
By applying (\ref{rel-cnt}) with $\eta=(t-p)/(t-2)$, we have
\begin{equation*}
D_{nt}\ll D_{np}\ll(D_{nt})^{(p-2)/(t-2)}. 
\end{equation*}
Then  (\ref{MD+LD}) holds for $0<x\leq C[\ln(D_{nt}^{-1})]^{1/2}$
with $C$ an arbitrary positive number. On the other hand,  there is a
constant $c_{1}>0$ such that for $x>c_{1}[\ln(D_{nt}^{-1})]^{1/2},$
we also have
\[
\mathbb{P}\left(  S_{n}\geq x\right)  =(1+o(1))\sum_{r,s\in \mathbb{Z}}%
\mathbb{P}(b_{n,r,s}\xi_{0}\geq x)
\]
and%
\[
1-\Phi(x)=o\bigg(\sum_{r,s\in \mathbb{Z}}\mathbb{P}(b_{n,r,s}\xi_{0}\geq x)\bigg).
\]
See also the proof of Theorem 2.1 in Peligrad et al. (2014b). By choosing
$c_{1}<C$, (\ref{MD+LD}) holds for all $x=x_n\rightarrow \infty$. If the
sequence $x_n$ is bounded, by Theorem \ref{ModerateD}, we have the moderate
deviation result. Since $x_n\ge c>0$, the second part in the right side of
(\ref{MD+LD}) is dominating as $n\rightarrow \infty$.
This finishes the proof.
\end{proof}


\section{Applications}

In this section, we provide two applications of the main results in Section 2,
one to nonparametric regression and the other to the Davis-Gut law for linear
random fields.

\subsection{Nonparametric regression}
We first provide an application of the deviation results in nonparametric
regression estimate. Consider the following regression model
$$
Y_{n,j,k}=g(z_{n,j,k})+X_{n,j,k}, \quad (j,k)\in \Gamma_n,
$$
where $g$ is a bounded continuous function on $\mathbb{R}^d$,
$z_{n, j,k}$'s are the fixed design points over $\Gamma_n \subseteq
\mathbb{Z}^2$ with values in a compact subset of $\mathbb{R}^d$, and
$X_{n,j,k}= \sum_{r,s\in\mathbb{Z}}a_{r,s}\xi_{n, j-r,k-s}$ is a linear random
field over $\mathbb{Z}^2$ with mean zero i.i.d. innovations $\xi_{n,r,s}$.
Regression models with independent or weakly dependent random field errors
have been studied by several authors including El Machkoui (2007), 
El Machkouri and Stoica (2010), 
Hallin et al. (2004a). 
For related papers that deal with density estimations for random fields,
see for example Tran (1990), Hallin et al. (2004b).


An estimator for the function $g$ on the basis of sample pairs
$(z_{n,j,k}, Y_{n,j,k})$, $(j,k)\in \Gamma_n$,
is the following general linear smoother:
\begin{equation*}\label{Eq:K}
g_n(z)=\sum_{(j,k)\in\Gamma_n} w_{n,j,k}(z)Y_{n, j,k},
\end{equation*}
where $w_{n,j,k}(\cdot)$'s are weight functions on $\mathbb R^d$.
In the particular case of kernel regression estimation,
$w_{n,j,k}(z)$ has the form
$$
w_{n,j,k}(z)=\frac{K(\frac{z-z_{n,j,k}}{h_n})}
{\sum_{(j',k')\in\Gamma_n} K(\frac{z-z_{n,j',k'}}{h_n})},
$$
where $K: \mathbb{R}^d\rightarrow \mathbb{R}^+$ is a kernel function
and $h_n$ is a sequence of bandwidths which goes to zero as $|\Gamma_n|
\rightarrow \infty$. Gu and Tran (2009) developed central limit theorem
and the bias for the fixed design regression estimate $g_n(z)$ in the
case when $X_{n,j,k}=\xi_{n,j,k}$ for all $(j,k)\in \Gamma_n$.

Theorems 2.1-2.3 in Section 2 can be applied to study, for every
$z \in {\mathbb R}^d$, the speed of the a.s. convergence of
$g_n(z) -\mathbb E g_n(z) \to 0$, or $g_n(z) \to g(z)$ if the weight
functions are chosen to satisfy the condition $\sum_{(j,k)\in\Gamma_n}
w_{n,j,k}(z)  = 1$, which is the case in kernel regression estimation.

Let $S_n:=g_n(z)-\mathbb{E}g_n(z)$. Then it can be written as
\begin{equation*}\label{res}
S_n = \sum_{(j,k)\in\Gamma_n} w_{n,j,k}(z)X_{n, j,k}
=\sum_{r,s\in\mathbb{Z}} b_{n,r,s}\xi_{n,-r, -s},
\end{equation*}
where $b_{n,r,s}=\sum_{(j,k)\in\Gamma_n} w_{n,j,k}(z) a_{j+r, k+s}$. We
choose the weight functions $w_{n,j,k}(z)$, the coefficients
$\{a_{r,s}\}$ and the random variable $\xi_0$ to satisfy the conditions
of Theorems \ref{ModerateD}. Then for $x_n = \sqrt{ 2\ln (U_{np}^{-1})}$,
(see Section 2 for the definition of $U_{np}$), we have
\begin{align*}
\mathbb{P}\big( |S_{n}|\geq x_{n}\sigma_{n}\big)  = 2(1-\Phi(x_{n}))
(1+o(1))\quad \text{ as }n\rightarrow\infty,
\end{align*}
where $\sigma_n^2=\sum_{r, s\in\mathbb{Z}}b_{n,r,s}^2={\rm Var}(g_n(z))$.

If $U_{np} \to 0$ is fast enough such that $\sum_n(1-\Phi(x_{n}))<\infty$,
then we can derive by using the Borel-Cantelli lemma the following upper
bound on the speed of convergence of $g_n(z)-\mathbb{E}g_n(z)$.
\begin{equation}\label{LIL1}
\limsup_{n\rightarrow \infty}\frac{|g_n(z)-\mathbb{E}g_n(z)|}
{\sigma_n\sqrt{2\ln (U_{np}^{-1}) }}
\le 1, \qquad \hbox{ a.s.}
\end{equation}
Under further conditions, we may put (\ref{LIL1}) in a more familiar form.
Since $p > 2$, we have $U_{np} \le |\rho_n|^{p-2}$. If we have
information on the rate for $\rho_n \to 0$, say,
$|\rho_n| \le (\ln n)^{-1/(p-2)}$, then we obtain an upper bound
which coincides with the law of the iterated logarithm:
\begin{equation}\label{LIL2}
\limsup_{n\rightarrow \infty}\frac{|g_n(z)-\mathbb{E}g_n(z)|}
{\sigma_n\sqrt{2\ln \ln n}}
\le1, \qquad \hbox{ a.s.}
\end{equation}

In the particular case of $X_{n,j,k}=\xi_{n,j,k}$,  $b_{n,r,s}
=w_{n,r,s}(z)$ if $(r,s)\in \Gamma_n$ and, otherwise $b_{n,r,s}=0$,
we have
$D_{nt}=\sum_{(r,s)\in \Gamma_n} w_{n,r,s}^t(z)$, $U_{nt}
=(D_{n2})^{-t/2}D_{nt}$
and
$$
\sigma_n^2=\mathbb{E}(S_n^2)=D_{n2}=\sum_{(r,s)\in \Gamma_n} w_{n,r,s}^2(z).
$$
Hence, under certain conditions on the weight functions $w_{n,j,k}(z)$,
we can obtain from (\ref{LIL1}) or (\ref{LIL2}) the speed of convergence
of $g_n(z)- \mathbb{E}g_n(z) \to 0$, which compliments the results in
Gu and Tran (2009).



\subsection{A Davis-Gut law of the iterated logarithm}

Now we apply the moderate deviation result, Theorem \ref{ModerateD},
to prove a Davis-Gut type law for linear random fields. See Davis (1968),
Gut (1980), Li (1991) and Li and Rosalsky (2007) for the Davis-Gut laws
for partial sums of i.i.d. random variables. The Davis-Gut type law for 
linear processes with short memory (short-range dependence)
was  developed in Chen and Wang (2008).

For a linear random field defined in (\ref{rf}), we consider the partial sum (\ref{Def:S}) 
with $\Gamma_n=[1,n]^2\cap \mathbb{Z}^2$, and assume the following condition: 
\begin{itemize}
\item[(DG)]\, $ \|\xi_0\|_p<\infty$ for some $p>2$ and $\{a_{r,s}\}$ satisfies either 
\begin{align}\label{con1}
A:=\sum_{r,s\in\mathbb{Z}}|a_{r,s}|<\infty,  \;\;a:=\sum_{r,s\in\mathbb{Z}}a_{r,s}\ne 0,
\end{align}
or 
\begin{equation}\label{la}
a_{r,s}=(|r|+|s|)^{-\beta}L(|r|+|s|) b\Big(\frac{r} {\sqrt{r^2+s^2}},\, \frac{ s} {\sqrt{r^2+s^2}}\Big)
\end{equation}
for $r\ne 0$ or $s\ne 0$, where $\beta \in(1,2)$, $L(\cdot)$ is a slowly varying function  
at infinity,  $b(\cdot, \cdot)$ is a bounded piece-wise continuous function defined on the 
unit circle.
\end{itemize}

Under the condition (\ref{la}), $\sum_{r,s\in\mathbb{Z}}|a_{r,s}|=\infty$. In the literature,
the random field (\ref{rf}) is said to have long memory or long range dependence. The 
following lemma gives the order of the quantity $D_{np}$ (see the definition in Section 2) 
under the condition (\ref{la}). Recall that $a_{n}\propto b_n$
means that $a_{n}/b_{n}\rightarrow C$ as $n\rightarrow \infty$
for some constant $C>0$. 

\begin{lemma}\label{Dorder}
Assume (\ref{la}), then  for $p>2$, 
$$D_{np}=\sum_{r, s\in \mathbb{Z}}|b_{n,r,s}|^p=O \big(n^{p(2-\beta)+2}L^p(n)\big).$$ 
\end{lemma}
\begin{proof}
We use the properties of slowly varying functions as stated in Lemma \ref{Karamata}, 
the condition $1<\beta<2$ and thus $1-\beta>-1$ and $1-p\beta<-1$ throughout the proof.
We also use $C>0$ as a generic constant in the proof.  
First we consider the case $r>n$. Since $b(\cdot, \cdot)$ is bounded,
\begin{align}
|b_{n,r,s}|&\le C\sum_{j, k=1}^n(j+r+|k+s|)^{-\beta}L(j+r+|k+s|)\label{r>n}\\
&\propto n\sum_{k=1}^n(r+|k+s|)^{-\beta}L(r+|k+s|)\notag\\
&\propto n^2(r+|s|)^{-\beta}L(r+|s|).\notag
\end{align}
Then 
\begin{align}
&\sum_{s\in \mathbb{Z}, r>n}|b_{n,r,s}|^p=\sum_{|s|\le n, r>n}|b_{n,r,s}|^p+\sum_{|s|> n, r>n}|b_{n,r,s}|^p\notag\\
&\ \ \le Cn\sum_{r>n}n^{2p}r^{-p\beta}L^p(r)+2C\sum_{s> n, r>n}n^{2p}(r+s)^{-p\beta}L^p(r+s)\notag\\
&\ \ \propto n^{2p+1}n^{1-p\beta}L^p(n)+\sum_{r>n}n^{2p}(r+n)^{1-p\beta}L^p(r+n)\notag\\
&\ \ \propto n^{p(2-\beta)+2}L^p(n)+n^{2p}n^{2-p\beta}L^p(n)\notag\\
&\ \ =2n^{p(2-\beta)+2}L^p(n).\label{term1}
\end{align}
For the case $r<-2n$, let $R=-r-n$, then $R>n$ and  
\begin{align*}
|b_{n,r,s}|&\le C\sum_{j, k=1}^n(-j-r+|k+s|)^{-\beta}L(-j-r+|k+s|)\\
&=C\sum_{j, k=1}^n(-j+n+R+|k+s|)^{-\beta}L(-j+n+R+|k+s|)\\
& \propto n^2(|r|+|s|)^{-\beta}L(|r|+|s|).  
\end{align*}
Hence, similarly to \eqref{term1}, we have
\begin{align}
\sum_{s\in \mathbb{Z}, r<-2n}|b_{n,r,s}|^p=O \big(n^{p(2-\beta)+2}L^p(n)\big).
\end{align}
By symmetry, we also have 
\begin{align}
\sum_{r\in \mathbb{Z}, s>n}|b_{n,r,s}|^p=O\big(n^{p(2-\beta)+2}L^p(n)\big), 
\end{align}
and 
\begin{align}
\sum_{r\in \mathbb{Z}, s<-2n}|b_{n,r,s}|^p=O\big(n^{p(2-\beta)+2}L^p(n)\big).
\end{align}
In the case $-2n\le r, s\le n$, 
\begin{align*}
|b_{n,r,s}|&\le C\sum_{j, k=1}^n(|j+r|+|k+s|)^{-\beta}L(|j+r|+|k+s|)\\
&\le 4C\sum_{j, k=1}^{2n}(j+k)^{-\beta}L(j+k)\\
&\ll \sum_{j=1}^{2n}j^{1-\beta} \max_{1\le k\le 2n}L(j+k)\\
&\propto n^{2-\beta} \max_{1\le k\le 2n}L(2n+k)\propto n^{2-\beta} L(n).
\end{align*}
Hence, 
\begin{align}\label{term5}
\sum_{-2n\le s, r\le n}|b_{n,r,s}|^p\ll n^2n^{p(2-\beta)}L^p(n)
=n^{p(2-\beta)+2}L^p(n).
\end{align}
Putting (\ref{term1})-(\ref{term5}) together, we complete the proof of the lemma. 
\end{proof}

The theorem below gives a
Davis-Gut type law for linear random fields that satisfy condition (DG).
\begin{theorem}\label{DG}
Assume condition (DG).  Let $h(\cdot)$ be a positive nondecreasing function on
$[c,\infty)$ for some constant  $c\ge 1$,
such that $\int_c^\infty (th(t))^{-1}dt=\infty$. Let $\Psi(t)=\int_c^t
(sh(s))^{-1}ds$, $t\ge c$. Let
$m=\arg\min_{t\ge c, t\in\mathbb{N}}\{\Psi(t)>1\}$. Then for real numbers
$\varepsilon$ and $n\ge m$,
we have
\begin{align*}
 \P \left(|S_n|>(1+\varepsilon)\sigma_n\sqrt{2\ln\Psi(n)}\right)
\propto \frac{1}{\sqrt{\ln\Psi(n)}}\Psi(n)^{-(1+\varepsilon)^2}.
\end{align*}
Define
\begin{align*}
S_{\Psi}:=\sum_{n=m}^\infty\frac{1}{nh(n)} \P\left(|S_n|>(1+\varepsilon)
\sigma_n\sqrt{2\ln\Psi(n)}\right).
\end{align*}
Then  $S_{\Psi}<\infty$ if $\varepsilon>0$ and $S_{\Psi}=\infty$
if $\varepsilon\le 0$.
\end{theorem}
\begin{proof}
First we consider the short memory case (\ref{con1}). Recall that $a=\sum_{r,s\in\mathbb{Z}}
a_{r,s}\ne 0$. Under condition (\ref{con1}), since 
$$
\sigma_n^2=\sum_{r,s\in\mathbb{Z}} b_{n,r,s}^2 =\sum_{r,s\in\mathbb{Z}}
\bigg(\sum_{j=0}^n\sum_{k=0}^n a_{j+r, k+s}\bigg)^2, 
$$
it is easy to see that $\sigma_n^2/n^2-a^2\rightarrow 0$. Hence $a^2n^2/\sigma_n^2\rightarrow 1$ as $n\rightarrow 
\infty$. Also the numbers $b_{n,r,s}$ that satisfy $|b_{n,r,s}|\ge 1$ are 
at most $\lceil a^2n^2\rceil$ asymptotically. Then
\begin{align*}
D_{np}&=\sum_{r,s\in\mathbb{Z}} |b_{n,r,s}|^p=\sum_{r,s\in\mathbb{Z},|b_{n,r,s}|\ge 1}
|b_{n,r,s}|^p
+\sum_{r,s\in\mathbb{Z},|b_{n,r,s}|< 1} |b_{n,r,s}|^p\notag\\
&\le \sum_{r,s\in\mathbb{Z},|b_{n,r,s}|\ge 1} A^p+\sum_{r,s\in\mathbb{Z},|b_{n,r,s}|< 1} b_{n,r,s}^2\notag\\
&\le \lceil a^2n^2\rceil A^p+\sigma_n^2
\end{align*}
has order $O(n^2)$ for $p>2$. Therefore $\ln(U_{np}^{-1})=\ln( \sigma_n^p/D_{np})
\ge (p-2)\ln n$.

Next we study the long memory case. Under condition (\ref{la}), Lemma \ref{Dorder} gives 
$D_{np}=O(n^{p(2-\beta)+2}L^p(n))$. On the other hand, by Theorem 2 of Surgailis (1982), 
\begin{align*}
\sigma_n^2=\sum_{r,s\in\mathbb{Z}} b_{n,r,s}^2=c_\beta n^{6-2\beta}L^2(n)
\end{align*}
for some constant $c_\beta$ depending only on $\beta$. Hence we also have 
$\ln \big(U_{np}^{-1}\big)=\ln \big(\sigma_n^p/D_{np}\big)
\ge (p-2)\ln n$. 

By the definition of $\Psi(t)$,  $\Psi(n) \le \int_c^n (sh(c))^{-1}ds\le\ln n/h(c)$. Let  
$x_n=(1+\varepsilon)\sqrt{2\ln\Psi(n)}$. Then  $x_n^2\ll2(p-2)\ln n\le 2\ln(U_{np}^{-1})$. 
By Remark \ref{remark1}, $\sigma_n\rightarrow\infty$ implies that
 $\rho_n\rightarrow 0$ in our case.
 Then by Theorem \ref{ModerateD},
 \begin{align}
&\mathbb P \left(|S_n|>(1+\varepsilon)\sigma_n\sqrt{2\ln\Psi (n)} \right)\notag\\
&=2 \left(1-\Phi(x_n)\right)\left(1+o(1)\right)\notag\\
&=2(2\pi)^{-1/2}\frac{1}{(1+\varepsilon)\sqrt{2\ln\Psi (n)}}
\exp \left(-(1+\varepsilon)^2
\ln\Psi (n)\right)(1+o(1))\label{approx}\\
&\propto \frac{1} {\sqrt{\ln\Psi (n)}} \big(\Psi (n) \big)^{-(1+\varepsilon)^2}.\notag
\end{align}
In (\ref{approx}) we have used the well-known inequality
$$
\frac{1}{(2\pi)^{1/2} (1+x)} \exp\Big(-\frac{x^{2}} 2 \Big)\le 1-\Phi(x)\le \frac 1 {(2\pi)^{1/2}
x}\exp\Big(- \frac{x^2} 2\Big), \ \hbox{ for } \, x > 1.
$$
Therefore,
\begin{align*}
S_{\Psi}&=\sum_{n=m}^\infty\frac{1}{nh(n)} \mathbb P \left(|S_n|>(1+\varepsilon)
\sigma_n\sqrt{2\ln\Psi(n)}\right)\notag\\
&\propto \sum_{n=m}^\infty\frac{1}{nh(n)\sqrt{\ln\Psi(n)}
\Psi(n)^{(1+\varepsilon)^2}}\notag\\
&=\sum_{n=m}^\infty\frac{\Psi'(n)}{\sqrt{\ln\Psi(n)}
\Psi(n)^{(1+\varepsilon)^2}}.\notag
\end{align*}
It is clear that $S_{\Psi}<\infty$ if $\varepsilon>0$ and $S_{\Psi}=\infty$ 
if $\varepsilon\le 0$.
\end{proof}

\begin{corollary}\label{cor1}
Assume condition (DG). Let
\begin{align*}
S=\sum_{n=3}^\infty\frac{1}{n(\ln\ln n)^b}
\mathbb P \left(|S_n|>(1+\varepsilon)\sigma_n\sqrt{2\ln\ln n}\right).
\end{align*}
Then for any $b\in\mathbb{R}$,
$S<\infty$ if $\varepsilon>0$ and $S=\infty$ if $\varepsilon<0$.
If $\varepsilon=0$,  $S<\infty$ if $b>\frac{1}{2}$ and $S=\infty$ if 
$b\le\frac{1}{2}$.
\end{corollary}
\begin{proof}
Let $h(t)=1$ and $c=1$.  Then $\Psi(n)=\ln n$.
By Theorem \ref{DG}, for $n\ge 3$,
\begin{align}
\mathbb P\left(|S_n|>(1+\varepsilon)\sigma_n\sqrt{2\ln\ln n}\right)
\propto \frac{1}{\sqrt{\ln\ln n}}(\ln n)^{-(1+\varepsilon)^2}.\notag
\end{align}
For any $b\in \mathbb{R}$,
\begin{align}
S&=\sum_{n=3}^\infty\frac{1}{n(\ln\ln n)^b}\mathbb P \left(|S_n|>(1+\varepsilon)
\sigma_n\sqrt{2\ln\ln n}\right)\notag\\
&\propto \sum_{n=3}^\infty\frac{1}{n(\ln\ln n)^{b+1/2}}
(\ln n)^{-(1+\varepsilon)^2}.\label{ep0}
\end{align}
It is clear that $S<\infty$ if $\varepsilon>0$ and $S=\infty$ 
if $\varepsilon<0$. In the case $\varepsilon=0$, by (\ref{ep0}), it is easy to see that 
$S<\infty$ if $b>\frac{1}{2}$ and $S=\infty$ if 
$b\le\frac{1}{2}$.
\end{proof}

\begin{corollary}\label{cor2}
Assume condition (DG).  For $0\le r<1$, let
\begin{align*}
S_r&=\sum_{n=3}^\infty\frac{1}{n(\ln n)^r}
\mathbb P\left(|S_n|>(1+\varepsilon)\sigma_n\sqrt{2(1-r)\ln\ln n}\right).
\end{align*}
Then
$S_r<\infty$ if $\varepsilon>0$ and $S_r=\infty$ if $\varepsilon\le 0$.
\end{corollary}
\begin{proof}
Let $h(t)=(\ln t)^r/(1-r)$, $c=1$. Then $\Psi(n)=(\ln n)^{1-r}$.
By Theorem \ref{DG},    for $n\ge 3$,
\begin{align}
\mathbb P\left(|S_n|>(1+\varepsilon)\sigma_n\sqrt{2(1-r)\ln\ln n}\right)
\propto \frac{1}{\sqrt{\ln\ln n}}(\ln n)^{-(1+\varepsilon)^2(1-r)
}\notag
\end{align}
and
\begin{align*}
S_r&=\sum_{n=3}^\infty\frac{1}{n(\ln n)^r}\mathbb P\left(|S_n|>
(1+\varepsilon)\sigma_n\sqrt{2(1-r)\ln\ln n}\right)\notag\\
&\propto \sum_{n=3}^\infty\frac{1}{n\sqrt{\ln\ln n}(\ln n)^{
(1+\varepsilon)^2(1-r)+r}}\notag\\
&=\sum_{n=3}^\infty\frac{1}{n\sqrt{\ln\ln n}(\ln n)^{1+(2\varepsilon+\varepsilon^2)
(1-r)}}.\notag
\end{align*}
It is clear that $S_r<\infty$ if $\varepsilon>0$ and $S_r=\infty$ 
if $\varepsilon\le 0$.
\end{proof}

\begin{corollary}\label{cor3}
Assume condition (DG). Let
\begin{align*}
S=\sum_{n=16}^\infty\frac{1}{n\ln n} \mathbb P\left(|S_n|>(1+\varepsilon)
\sigma_n\sqrt{2\ln\ln\ln n}\right).
\end{align*}
Then $S<\infty$ if $\varepsilon>0$ and $S=\infty$ if $\varepsilon\le 0$.
\end{corollary}
\begin{proof}
Let $h(t)=\ln t$, $c=e$. Then $\Psi(n)=\ln\ln n$. By Theorem \ref{DG},    
for $n\ge 16$,
\begin{align}
\mathbb P\left(|S_n|>(1+\varepsilon)\sigma_n\sqrt{2\ln\ln\ln n}\right)
\propto \frac{1}{\sqrt{\ln\ln\ln n}}(\ln\ln n)^{-(1+\varepsilon)^2}\notag
\end{align}
and
\begin{align*}
S&=\sum_{n=16}^\infty\frac{1}{n\ln n}\mathbb P\left(|S_n|>(1+\varepsilon)
\sigma_n\sqrt{2\ln\ln\ln n}\right)\notag\\
&\propto \sum_{n=16}^\infty\frac{1}{n\ln n(\ln\ln n)^{(1+\varepsilon)^2}
\sqrt{\ln\ln\ln n}}.\notag
\end{align*}
It is clear that $S<\infty$ if $\varepsilon>0$ and $S=\infty$ if $\varepsilon\le 0$.
\end{proof}
\begin{remark}
One can also prove the Davis-Gut law for linear processes by applying the
moderate deviation results for linear processes in Peligrad et al. (2014a).  
Let $S_{n}=\sum_{k=1}^{n}X_{k}$,
where
\begin{equation*}
X_{k}=\sum_{j=-\infty}^{\infty}a_{k-j}\xi_{j} \label{ln}%
\end{equation*}
and the innovations $\xi_j$ are i.i.d. random variables with $\mathbb{E}\xi_j=0$ 
and $\mathbb{E}\xi_j^2=1$. Consider the short memory case $\sum_{i= 
-\infty}^\infty |a_i|<\infty, a=\sum_{i=-\infty}^\infty a_i\ne 0$.
Observe that $S_{n}=\sum_{i=-\infty}^{\infty}b_{ni}\xi_{i}$ where
$b_{ni}=a_{1-i}+\cdots+a_{n-i}$.  Then $\sigma_n^2={\rm Var}(S_n)=
\sum_{i}b_{ni}^{2}$. For the short memory case,
it is well known that $\sigma_n^2$ has order $n$. Furthermore,  
$a^2n/\sigma_n^2\rightarrow 1$ as $n\rightarrow\infty$; $\sum_i |b_{ni}|^p$ 
has order $n$ for $p>2$. Let
\begin{equation*}
U_{np}=\bigg(\sum_{i}b_{ni}^{2}\bigg)^{-p/2}\sum_{i}|b_{ni}|^{p}. 
\end{equation*}
Then $\ln(U_{np}^{-1})=\ln[(\sum_{i}b_{ni}^{2})^{p/2}/\sum_{i}|b_{ni}|^{p}]
\sim \frac{1}{2}(p-2)\ln n$.
Let $h(t)$ and $\Psi(t)$ be the functions defined as in Theorem \ref{DG}. 
Hence
\begin{align}\label{x_n}
x_n=(1+\varepsilon)\sqrt{2\ln\Psi(n)}
\ll \sqrt{(p-2)\ln n}\sim \sqrt{2\ln(U_{np}^{-1})}.
\end{align}
Then by part (iii) of Corollary 3 in  Peligrad et al. (2014a),  the Davis-Gut type laws,
Theorem \ref{DG} and Corollary \ref{cor1}, \ref{cor2}, \ref{cor3},  hold for short
memory linear processes. 
\end{remark}

\begin{remark}
For the causal long memory linear process, $X_{k}=\sum_{j=0}^{\infty}
a_{k-j}\xi_{j}$, with $\sum_{i=0}^\infty |a_i|=\infty$,
$\sum_{i=0}^\infty a_i^2<\infty$,  we assume that $a_{n}=(n+1)^{-\alpha}L(n+1)$,
where $1/2<\alpha<1$, and $L(n)>0$ is a slowly varying function at infinity.
Then $S_n=\sum_{i=1}^\infty b_{ni}\xi_{n-i}$ where $b_{ni}=\sum_{k=1}^i a_k$
for $i<n$ and $b_{ni}=\sum_{k=i-n+1}^i a_i$ for $i\ge n$. Also
\begin{equation} \label{sumb}
\sigma_n^2= {\rm Var}(S_n)=\sum_{i=1}^{\infty}b_{ni}^{2}\sim c_\alpha n^{3-2\alpha}
L^{2}(n),
\end{equation}
where
\begin{equation*}
c_{\alpha}= \frac 1 {(1-\alpha)^{2}} \int_{0}^{\infty}[x^{1-\alpha}-
\max(x-1,0)^{1-\alpha}%
]^{2}dx. \label{defcalpha}%
\end{equation*}

The asymptotic equivalence in (\ref{sumb}) is well known. See for instance
Theorem 2 in Wu and Min (2005). On the other hand, there are
constants $C_{1}$ and $C_{2}$ such that for all $n\geq1,$%
\begin{equation}
b_{ni}\leq C_{1}i^{1-\alpha}L(i)\text{ for }i\leq2n\text{ and }%
b_{ni}\leq C_{2}n(i-n)^{-\alpha}L(i)\text{ for }i>2n.\notag
\end{equation}
Hence, by the properties of slowly varying functions as stated in Lemma \ref{Karamata},
\begin{align*}
\sum_i b_{ni}^p&\ll\sum_{i<2n}i^{(1-\alpha)p}L^p(i)+ \sum_{i\ge 2n}n^p(i-n)^{-\alpha p}
L^p(i)\notag\\
&\ll n^{1+p(1-\alpha)}L^p(n).
\end{align*}
Therefore,
\begin{align*}
\ln(U_{np}^{-1})&=\ln \bigg[(\sum_{i}b_{ni}^{2})^{p/2}/\sum_{i}b_{ni}^{p}\bigg]\notag\\
&\gg\ln[n^{p(3-2\alpha)/2}L^{p}(n)/ (n^{1+p(1-\alpha)}L^p(n))]\notag\\
&\sim \frac{1}{2}(p-2)\ln n.
\end{align*}
Let $h(t)$ and $\Psi(t)$ be the functions defined as in Theorem \ref{DG}. Hence (\ref{x_n}) still holds. 
Then by Corollary 3, part (iii) of Peligrad et al. (2014a),
the Davis-Gut type laws,
Theorem \ref{DG} and Corollary \ref{cor1}, \ref{cor2}, \ref{cor3},  hold for
long memory linear processes.

\end{remark}

\section{Appendix}
In the Appendix, we first justify (1) and (2) in the Introduction, and then collect some results that are useful 
for the proofs in Section 2.

\begin{lemma}\label{Lem:Conv}
Let $\{\xi_{r,s}, (r, s) \in \mathbb{Z}^2\}$ and
$\xi_0$ be i.i.d. random variables with $\mathbb{E}\xi_0=0$
and $\mathbb{E}\xi^2_0=1$, and let $\{a_{r,s}, (r, s) \in \mathbb{Z}^2\}$ is a square summable sequence
of constants. Then the following statement hold:
\begin{itemize}
\item[(i).] The series $\sum_{r, s\in \mathbb{Z}} a_{r,s}\xi_{j-r, k-s} $ converges in $L^2(\Omega, \mathbb P)$ 
and almost surely.
\item[(ii).] Equation (2) holds in $L^2(\Omega, \mathbb P)$ and almost surely.
\end{itemize}
\end{lemma}
\begin{proof} We refer to the series $\sum_{r, s\in \mathbb{Z}} a_{r,s}\xi_{j-r, k-s} $ by (1). Let  $\{\Upsilon_n, n \ge 1\}$
be an arbitrary sequence  of finite subsets of $\mathbb Z^2$ that satisfy $\Upsilon_n \subset \Upsilon_{n+1}$ and 
$|\Upsilon_n| \to \infty$. Then for any $m < n$,
\[
{\mathbb E}\left[\bigg(\sum_{(r, s)\in \Upsilon_{n}\backslash \Upsilon_m} a_{r,s}\xi_{j-r, k-s} \bigg)^2\right] = 
\sum_{(r, s)\in \Upsilon_{n}\backslash \Upsilon_m} a_{r,s}^2,
\]
which tends to 0 as $ m \to \infty$. This implies that (1) converges in $L^2(\Omega, \mathbb P)$.

Since the summands in series (1) 
are independent random variables, the almost sure 
convergence of (1) follows from Kolmogorov's Three-Series Theorem. Alternatively,
it follows from L\'evy's Equivalence Theorem which says that the almost sure convergence is equivalent
to convergence in probability or in law. 

 
Next let $n\ge 1$ be fixed and we write the partial sum $S_n$ as 
\begin{align*}
S_n 
&=\sum_{(j,k)\in\Gamma_n}  \sum_{r, s\in \mathbb Z} a_{r,s}\xi_{j-r, k-s}\\
&= \sum_{(j,k)\in\Gamma_n}  \sum_{r, s\in \mathbb Z} a_{j+r,k+s}\xi_{-r, -s}\\
&= \sum_{(j,k)\in\Gamma_n}  \lim_{m \to \infty} \sum_{r, s\in [-m, m]} a_{j+r,k+s}\xi_{-r, -s}\\
 &=\lim_{m \to \infty} \sum_{(j,k)\in\Gamma_n} \sum_{r, s\in [-m, m]} a_{j+r,k+s}\xi_{-r, -s},
\end{align*}
where the limit is taken either in $L^2(\Omega, \mathbb)$ or in  the almost sure sense. 
Since both index sets $\Gamma_n$ and $[-m, m]^2$ are finite,
we change the order of summation to get
\begin{align*}
S_n &= \lim_{m \to \infty} \sum_{r, s\in [-m, m]} \bigg(\sum_{(j,k)\in\Gamma_n}  a_{j+r,k+s}\bigg)\xi_{-r, -s}\\
&=\sum_{r, s\in \mathbb{Z}}\bigg(\sum_{(j,k)\in\Gamma_n}  a_{j+r,k+s}\bigg)\xi_{-r, -s}
=\sum_{r,s\in\mathbb{Z}} b_{n,r,s}\xi_{-r, -s}.
\end{align*}
This verifies (ii).
\end{proof}

The following theorem is an extended version of the Fuk--Nagaev inequality (see
Corollary 1.7 in Nagaev (1979)) for a double sum of infinite many random variables. See also the extension of Fuk--Nagaev inequality for sum of infinite many random variables,  
Theorem 5.1 and Remark 5.1 in Peligrad et al. (2014b).

\begin{theorem}\label{FN}
Let  $(X_{ni})_{i\in \mathbb{N}^2}$ be a set of
independent random variables with mean 0.  For a constant $m \ge 2$, let $\beta=m/(m+2)$
and $\alpha=1-\beta=2/(m+2)$. For  any $y>0$, define $X_{ni}^{(y)}=X_{ni}I(X_{ni}\leq
y)$, $A_{n}(m;0,y):=\sum_{i\in\mathbb{N}^2}\mathbb{E}[X_{ni}^{m}I(0<X_{ni}<y)]$ and
$B_{n}^{2}(-\infty,y):=\sum_{i\in\mathbb{N}^2}\mathbb{E}[X_{ni}^{2}I(X_{ni}<y)].$
Then for any $x>0$ and $y>0$,
\begin{equation*}
\mathbb{P}\bigg(\sum_{i\in\mathbb{N}^2}X_{ni}^{(y)}\geq x \bigg)\leq\exp \bigg(-\frac{\alpha^{2}
x^{2}}{2e^{m}B_{n}^{2}(-\infty,y)} \bigg)+\bigg (\frac{A_{n}(m;0,y)}{\beta
xy^{m-1}}\bigg)^{\beta x/y}.
\end{equation*}
\end{theorem}

The following result is Theorem 5.2 of Peligrad et al. (2014b),
which is an immediate consequence of Theorem 1.1 in Frolov (2005).

\begin{theorem}
\label{frolov} Let $(X_{nj})_{1\leq j\leq k_{n}}$ be an array of row-wise
independent centered random variables. Let $S_{n}=\sum
_{j=1}^{k_{n}}X_{nj}$ and  $\sigma_n^2=\sum_{j=1}^{k_n}\mathbb{E}X_{nj}^2$.
For any positive numbers $u, v$ and $\varepsilon$, denote
\[
\Lambda_{n}(u,v,\varepsilon)=\frac{u}{\sigma_{n}^{2}}\sum_{j=1}^{k_{n}}
\mathbb{E} \big[X_{nj}^{2}I(X_{nj}\leq-\varepsilon\sigma_{n}/v)\big].
\]
Assume that for some constant  $p>2$,  $M_{np} =\sum_{j=1}^{k_{n}}\mathbb{E} \big[X_{nj}^{p}
I(X_{nj}\geq0)\big]<\infty$ and  $L_{np} :=\sigma_{n}^{-p}M_{np} \rightarrow 0$  as $n \to \infty$.
If  $\Lambda_{n}(x^{4},x^{5},\varepsilon)\rightarrow0$ for any $\varepsilon>0$ and 
$x^{2} -2\ln(L_{np}^{-1})-(p-1)\ln\ln(L_{np}^{-1})\rightarrow-\infty$ as $n \to \infty$, then
\[
\mathbb{P}\left(  S_{n}\geq x\sigma_{n}\right)  =(1-\Phi(x))(1+o(1)).
\]
\end{theorem}

The following lemma is useful in the proof of Theorem \ref{mix}. It is
Proposition 5.1 in Peligrad et al. (2014b).
\begin{lemma}
\label{frolov-trunc} Assume the conditions in Theorem \ref{frolov} are
satisfied. Fix $\varepsilon>0.$ Define
\[
X_{nj}^{(\varepsilon x\sigma_{n})}=X_{nj}I(X_{nj}\leq\varepsilon x\sigma
_{n})\ \text{ and } \ \ S_{n}^{(\varepsilon x\sigma_{n})}=\sum_{j=1}^{k_{n}}
X_{nj}^{(\varepsilon x\sigma_{n})}.
\]
If $x^{2}\leq c\ln(L_{np}^{-1})$ with $c<1/\varepsilon$, then as $n
\to \infty$ we have
\[
\mathbb{P}\left(  S_{n}^{(\varepsilon x\sigma_{n})}\geq x\sigma_{n}\right)
=(1-\Phi(x))(1+o(1)).
\]
\end{lemma}

The following lemma lists some properties of the slowly varying
function. Their proofs can be found in Bingham et al. (1987) or Seneta (1976).

\begin{lemma}
\label{Karamata} A slowly varying function $l(x)$ defined on $[A,\infty)$ has
the following properties: 
\begin{enumerate}
\item For $A<c<C<\infty$, $\lim_{x\rightarrow\infty}\frac{l(tx)}{l(x)}=1$
uniformly in $c\leq t\leq C$.

\item For any $\theta>-1$, $\int_{A}^{x}y^{\theta}l(y)dy\mathbb{\sim}%
\frac{x^{\theta+1}l(x)}{\theta+1}$ as $x\rightarrow\infty$.

\item For any $\theta<-1$, $\int_{x}^{\infty}y^{\theta}l(y)dy\mathbb{\sim
}\frac{x^{\theta+1}l(x)}{-\theta-1}$ as $x\rightarrow\infty$.

\item For any $\eta>0$, $\sup_{t\geq x}(t^{\eta}l(t))\mathbb{\sim}x^{\eta
}l(x)$ as $x\rightarrow\infty$. Moreover, $\sup_{t\geq x}(t^{\eta}%
l(t))=x^{\eta}\bar{l}(x)$, where $\bar{l}(x)$ is slowly varying and $\bar
{l}(x)\mathbb{\sim}l(x).$
\end{enumerate}
\end{lemma}


\begin{thebibliography}{99}
\bibitem{BDP}
Banys, P., Davydov, Y. and Paulauskas, V., 2010.  Remarks on the SLLN for linear random fields. 
{\it Statist. Probab. Lett.} {\bf 80}, 489-496.

\bibitem {BinghamGoldieTeugels}
Bingham, N. H., Goldie, C. M.  and Teugels,  J. L., 1987. \textit{Regular Variation}. Cambridge University Press,
Cambridge, UK.



\bibitem{CW}
Chen, P. and Wang, D., 2008. Convergence rates for probabilities of moderate
deviations for moving average processes.  {\it Acta Math. Sinica, English Series}
{\bf 24}, 611-622.

\bibitem{Davis}
Davis, J. A., 1968. Convergence rates for the law of the iterated logarithm.  {\it Ann.
Math. Statist.} {\bf 39}, 1479-1485.

\bibitem{DavisHsing}
Davis, R. A. and Hsing, T. 1995.
Point process and partial sum convergence for weakly dependent random variables with infinite variance.
{\it Ann. Probab.} {\bf 23},  879-917.

\bibitem{delaPenaDine}
de la Pe\~na, V. and Gin\'e, E.,  1999. {\it Decoupling. From Dependence to Independence}. 
Springer, New York.

\bibitem{DG01}
Djellout, H. and Guillin, A., 2001.  Large and moderate deviations for moving average processes.
{\it Ann. Fac. Sci. Toulouse Math. } (6) {\bf 10},  23-31.

\bibitem{DGW06}
Djellout, H., Guillin, A. and Wu, L., 2006. Moderate deviations of empirical periodogram
and non-linear functionals of moving average processes. {\it Ann. Inst. H. Poincar\'e Probab. Statist.}
{\bf 42},  393-416.

\bibitem{ElM07}
 El Machkouri, M.,  2007.  Nonparametric regression estimation for random fields in a fixed-design.
 {\it Stat. Inference Stoch. Process.} {\bf 10},   29--47.

\bibitem{ElM14}
El Machkouri, M.,  2014. Kernel density estimation for stationary random fields. {\it
ALEA Lat. Am. J. Probab. Math. Stat.} {\bf 11}, 259-279.


\bibitem{ElMS10}
El Machkouri, M. and Stoica, R., 2010.  Asymptotic normality of kernel estimates in a regression
model for random fields. {\it J. Nonparametr. Stat.} {\bf 22}, 955-971.

\bibitem {frolov}
Frolov, A. N., 2005. On probabilities of moderate deviations
of sums for independent random variables. \textit{J. Math. Sci.} \textbf{127}, 1787-1796.

\bibitem {GhoshS09}
Ghosh, S. and Samorodnitsky, G., 2009. The effect of memory on functional large
deviations of infinite moving average processes. {\it Stoch. Process. Appl.} {\bf 119}, 534-561

\bibitem{GuTran}
Gu, W. and  Tran,  L. T.,  2009. Fixed design regression for negatively associated random fields.
\textit{J. Nonparametric Statist.} \textbf{21}, 345-363.

\bibitem{Gut}
Gut, A., 1980. Convergence rates for probabilities of moderate deviations for sums of
random variables with multidimensional indices. {\it Ann. Probab.} {\bf 8}, 298-313.


\bibitem{HLT04}
Hallin, M., Lu, Z.  and Tran, L. T.,  2004a.   Local linear spatial regression. {\it Ann. Statist.}
{\bf 32}, 2469-2500.

\bibitem{HLT}
Hallin, M., Lu, Z.  and Tran, L. T.,  2004b. Kernel density estimation for spatial processes:
the $L_1$ theory. \textit{J. Multivariate Anal.}   \textbf{88}, 61-75.

\bibitem{JRW95}
Jiang, T., Rao, M., and Wang, X., 1995. Large deviations for moving average processes,.
{\it  Stoch. Process. Appl. } {\bf  59}, 309--320.


\bibitem{Klesov}
Klesov, O., 2014. \textit{Limit Theorems
for Multi-Indexed Sums of Random Variables}.
Probability Theory and Stochastic Modelling \textbf{71}, Springer, Heidelberg New York Dordrecht London.

\bibitem{Li}
Li, D., 1991. Convergence rates of law of iterated logarithm for B-valued random
variables. {\it Sci. China Ser. A} {\bf 34}, 395-404.

\bibitem{LR}
Li, D. and Rosakly, A.,  2007. A supplement to the Davis-Gut law. {\it J. Math. Anal. Appl.}
{\bf 330}, 1488-1493.

\bibitem{LLX}
Li, L., J. Liu and Xiao, Y., 2009. Wavelet regression with long memory infinite
moving average errors. \textit{J. Appl. Probab. Stat.}   \textbf{2}, 183--211.


\bibitem{MallikWoodroofe}
Mallik, A. and Woodroofe, M.,  2011. A central limit theorem for linear random fields. \
 \textit{Statist. Probab. Lett.} \textbf{81} 1623-1626.

\bibitem{MarinucciPoghosyan}
Marinucci, D. and Poghosyan, S.,  2001. Asymptotics for linear random fields.
\textit{Statist. Probab. Lett.} \textbf{51}  131-141.

\bibitem{MS00}
Mikosch, T. and Samorodnitsky, G., 2000.
The supremum of a negative drift random walk with dependent heavy-tailed steps.  
{\it Ann. Appl. Probab.} {\bf 10}, 1025-1064. 

\bibitem{MW13}
Mikosch, T. and  Wintenberger, O., 2013. 
Precise large deviations for dependent regularly varying sequences.  
{\it Probab. Th. Rel. Fields} {\bf 156},   851-887. 

\bibitem {AVNagaev1}
Nagaev, A. V.,  1969a. Limit theorems for large deviations
where Cram\'{e}r's conditions are violated (in Russian).
\textit{Izv. Akad. Nauk UzSSR Ser. Fiz.-Mat. Nauk} \textbf{6},
17-22.

\bibitem {AVNagaev2}
Nagaev, A. V.,  1969b. Integral limit theorems for large deviations
when Cram\'{e}r's condition is not fulfilled I, II.
\textit{Theory Probab. Appl.} \textbf{14},
51-64, 193-208.

\bibitem {Nagaev2} 
Nagaev, S. V.,  1979. Large deviations of sums of
independent random variables. \textit{Ann. Probab.} \textbf{7}, 745-789.

\bibitem{Paul}
Paulauskas, M.,  2010. On Beveridge-Nelson decompositions and limit
theorems for linear random fields.
{\it J. Multivariate Anal.} {\bf 101}, 3621-3639.

\bibitem{PSZW}
Peligrad, M., Sang, H.,  Zhong, Y.  and Wu, W. B., 2014a. Exact moderate and large
deviations for linear processes.  \textit{Statist. Sinica} \textbf{24}, 957-969.

\bibitem{PSZW1}
Peligrad, M., Sang, H.,  Zhong, Y.  and Wu, W. B., 2014b.  Supplementary
material for the paper \textquotedblleft Exact moderate and large deviations
for linear processes". \textit{Statist. Sinica}, 15 pp, available online at
{\it http://www3.stat.sinica.edu.tw/statistica/}

\bibitem {SaulisS00}
 Saulis, L. and Statulevi\u cius, V., 2000. Limit theorems on large deviations. In:
{\it Limit Theorems of Probability Theory}. (Prokhorov, Yu. V. and Statulevi\u cius, V.,
editors), pp. 185--266, Springer, New York.

\bibitem{Seneta}
Seneta, E., 1976. \textit{Regularly Varying Functions}.
Lecture Notes in Mathematics \textbf{508}, Springer, Berlin.

\bibitem{Surgailis}
Surgailis, D.,  1982.  Zones of attraction of self-similar multiple integrals. 
 \textit{Lith Math J}  \textbf{22} 185-201. 

\bibitem{Tran90}
Tran, L. T.,  1990.  Kernel density estimation on random fields.
{\it J. Multivariate Anal.} {\bf 34}, 37-53.


\bibitem{WW14}
Wang, Y. and  Woodroofe, M.,  2014. On the asymptotic normality of kernel density
estimators for causal linear random fields. {\it J. Multivariate Anal.} {\bf 123}, 201-213.

\bibitem {WuMin05}
Wu, W. B. and Min, W.,  2005. On linear processes with dependent innovations. 
\textit{Stochastic Process. Appl.} \textbf{115}, 939-958.

\bibitem {WuZhao}
Wu, W. B. and Zhao, Z., 2008. Moderate deviations for
stationary processes. \textit{Statist. Sinica} \textbf{18}, 769-782.


\end{thebibliography}
\end{document}